\documentclass{article}
\usepackage{graphicx} 

\usepackage{amssymb}
\usepackage{amsthm}
\usepackage{extarrows}
\usepackage{calrsfs}
\usepackage{tikz-cd}
\usepackage{mathrsfs}
\usepackage[colorlinks = true,citecolor=blue,linkcolor=blue]{hyperref}

\usepackage{dynkin-diagrams}

\newtheorem{theorem}{Theorem}[section]
\newtheorem{definition}{Definition}[section]

\newtheorem{remark}{Remark}[section]

\newtheorem{coro}{Corollary}[section]
\newtheorem{example}{Example}[section]

\newtheorem{thmx}{Theorem}[section]

\newtheorem{Coro}{Corollary}[section]

\title{Chern-Ricci flow and $t$-Gauduchon Ricci-flat condition}

\author{Eder M. Correa \ \ Giovane Galindo \ \ Lino Grama}
\date{}

\begin{document}
\maketitle

\begin{abstract}
In this paper, we study the $t$-Gauduchon Ricci-flat condition under the Chern-Ricci flow. In this setting, we provide examples of Chern-Ricci flow on compact non-Kähler Calabi-Yau manifolds which do not preserve the $t$-Gauduchon Ricci-flat condition for $t<1$. The approach presented generalizes some previous constructions on Hopf manifolds. Also, we provide non-trivial new examples of balanced non-pluriclosed solution to the pluriclosed flow on non-Kähler manifolds. Further, we describe the limiting behavior, in the Gromov-Hausdorff sense, of geometric flows of Hermitian metrics (including the Chern-Ricci flow and the pluriclosed flow) on certain principal torus bundles over flag manifolds. In this last setting, we describe explicitly the Gromov-Hausdorff limit of the pluriclosed flow on principal $T^{2}$-bundles over the Fano threefold ${\mathbb{P}}(T_{{\mathbb{P}^{2}}})$.
\end{abstract}

\tableofcontents

\section{Introduction}

For a compact K\"{a}hler manifold $(M, g_0, J)$, the Kähler-Ricci flow starting at some K\"{a}hler metric $\omega_0 = g_0(J\cdot,\cdot)$ is defined by the PDE
\begin{align}\label{Eq1}
   \begin{cases}
       \displaystyle \frac{\partial}{\partial s}\omega(s) = -p_{\omega_s}, \\
       \omega(0) = \omega_0,
   \end{cases}
\end{align}
where $p_{\omega_s}$ is the Kähler-Ricci form of $\omega_{s}$. From the short-time existence result of Hamilton \cite{hamilton1982three}, and the fact that a maximal solution to the Ricci flow preserves the K\"{a}hler condition (e.g. \cite{MR1375255}), it follows that the initial-value problem (\ref{Eq1}) always admits a unique solution $\omega(s)$ defined on a maximal interval $[0,T)$, with $0 \leq T \leq +\infty$. Moreover, a result of Tian and Zhang \cite{tian2006kahler} gives a concrete characterization for the maximal existence time $T$.

\

If the initial manifold $(M, g_0, J)$ is not Kähler, the Ricci flow may not preserve the Hermitian condition $g(J\cdot ,J\cdot ) = g$. This occurs because the Levi-Civita connection is not compatible with the almost-complex structure $J$, making it less suitable for studying the complex properties of the manifold (e.g. \cite{tian2011hermitian}). To address this issue, it is common to consider Hermitian connections which are compatible with both the metric and the almost-complex structure, although they are not torsion-free in general. Among these connections, there is a distinguished one-parameter family of connections $\nabla^t$, $t \in \mathbb{R}$, defined by Gauduchon \cite{gauduchon1997Hermitian}, known as the canonical connections. The most well-known example is the Chern connection $\nabla^1$, which is characterized by the fact that its torsion has no $(1,1)$-component. All of these connections coincide with the Levi-Civita connection if the manifold is Kähler.

\

It is therefore natural to attempt to generalize the Kähler-Ricci flow to non-Kähler manifolds using these connections. In particular, we can modify the evolution equation by replacing the Kähler-Ricci form with the Chern-Ricci form $p(\omega, 1) = \sqrt{-1} \, \text{tr}(R(\nabla^1))$, which is the Ricci form for the Chern connection. This was first introduced in \cite{gill2011convergence} and later generalized for Hermitian manifolds by \cite{tosatti2015evolution}, where they considered the PDE
\begin{equation}\label{Eq2}
   \begin{cases}
       \displaystyle \frac{\partial}{\partial s}\omega(s)= -p(\omega(s),1), \\
       \omega(0) = \omega_0,
   \end{cases}
\end{equation}
which is essentially the same as Equation (\ref{Eq1}), except that the connection used is the Chern connection instead of the Levi-Civita connection. It was also proven in \cite{tosatti2015evolution} that there is a short-time solution to this flow, which yields a family of Hermitian metrics $\omega_s$. It was later proposed by \cite{chu2019monge} that, by considering only the $(1,1)$-part of the Chern-Ricci form in Equation (\ref{Eq2}), the Chern-Ricci flow could be extended to almost Hermitian manifolds.
\\

A natural question that arises in the above setting is what Hermitian properties are preserved by the Chern-Ricci flow. Of course, if the initial metric is Chern-Ricci flat—that is, if the Chern connection is Ricci-flat—then the flow is constant and this property is preserved. Other results concerned with the evolution of curvature conditions through the flow of Hermitian metrics can be found in \cite{tosatti2015evolution}, \cite{correa2023levi}, \cite{ustinovskiy2019hermitian}, \cite{yang2016chern}, \cite{ustinovskiy2021lie}.\\

In this paper, we show that the property of being $t$-Gauduchon Ricci-flat is not necessarily preserved for $t < 1$ under the Chern-Ricci flow. More precisely, for a suitable class of Hermitian non-K\"{a}hler manifolds $({\rm{U}}(E), J, g_0)$, considered in \cite{correa2023t}, where ${\rm{U}}(E)$ denotes the unitary frame bundle of a certain Hermitian holomorphic vector bundle $(E,h)$ over a complex flag manifold $(X,\omega_{0})$
, we prove the following result.
\begin{thmx}
\label{TeoremaA}
Let $X $ be a complex flag manifold with Picard number $\varrho(X)>1$ and equipped with a Kähler-Einstein metric $\omega_0 =\lambda p_{\omega_0}$. Then, there exists a Hermitian holomorphic vector bundle $(E,h) \to X$, such that:
\begin{enumerate}
\item[(a)] $E= \mathcal{O}_{X}(k) \oplus F_1 \oplus ... \oplus F_{2n-1}$, where $\mathcal{O}_{X}(k), F_1, \ldots, F_{2n-1} \in {\rm{Pic}}(X)$;
\item[(b)] The underlying complex manifold $({\rm{U}}(E), J)$ admits a family of Hermitian metrics $\Omega(s), s \in (-\infty,\lambda)$, solving the Chern-Ricci flow
\begin{equation}
\displaystyle \frac{\partial}{\partial s}\Omega(s)= -p(\Omega(s),1), \ \ \ \Omega(0) = \Omega_0,
\end{equation}
such that $({\rm{U}}(E),\Omega_{0}) \to (X,\omega_{0})$ is a Hermitian submersion;
\item[(c)] For each $s \in (-\infty , \lambda)$, we have that $({\rm{U}}(E),\Omega(s) , J)$ is $t$-Gauduchon Ricci-flat if, and only if,
\begin{equation*}
  t =1 - \frac{2(\lambda-s) {\rm{I}}(X)^2}{k^2 \dim_{\mathbb{C}}(X )},
\end{equation*}
where ${\rm{I}}(X)$ is the Fano index of $X$;
\item[(d)] ${\rm{U}}(E)$ does not admit any Kähler metric and, if $I(X)\lambda \in k\mathbb{Z}$, then 
\begin{equation}
c_{1}(T{\rm{U}}(E)) = 0 \in H^{2}({\rm{U}}(E),\mathbb{Z}), 
\end{equation}
i.e., ${\rm{U}}(E)$ is non-Kähler Calabi-Yau.
\end{enumerate}
 \end{thmx}

It is worth to point out that, in this paper, a non-Kähler Calabi-Yau manifold is a Hermitian manifold $X$ which does not admit any Kähler structure and satisfies $c_{1}(X) = 0 \in H^{2}(X,\mathbb{Z})$. The above theorem provides a source of examples of solutions to the Chern-Ricci flow on non-Kähler Calabi-Yau manifolds. Moreover, an immediate consequence of the item (c) of the above result is the following corollary.

\begin{Coro}
The Chern-Ricci flow does not preserve the $t$-Gauduchon Ricci-flat condition.
\end{Coro}

An important class of Hermitian metrics also under consideration in this work is the class of balanced Hermitian metrics. A Hermitian metric $g$ is called balanced if $\omega = g(J\cdot,\cdot)$ is co-closed, meaning $\delta \omega = 0$. Originally introduced by Michelsohn in \cite{10.1007/BF02392356}, balanced metrics have played a significant role in Hermitian geometry, see for instance  \cite{fino2023balanced}, \cite{fino2019astheno}, \cite{fino2016existence}, \cite{article}, and \cite{alessandrini1995modifications}. \\

In the setting of balanced Hermitian metrics, our next result provides a solution to the Chern-Ricci flow defined by a family of balanced non-K\"{a}hler Hermitian metrics. More precisely, we have the following result.

\begin{thmx}
\label{TeoremaB}
Let $X $ be a complex flag manifold with Picard number $\varrho(X)>1$ and equipped with a Kähler-Einstein metric $\omega_0 =\lambda p_{\omega_0}$. Then, there exists a Hermitian holomorphic vector bundle $(E,h) \to X$, such that:
\begin{enumerate}
\item[(a)] $E=  F_1 \oplus ... \oplus F_{2n}$, where $F_1, \ldots, F_{2n} \in {\rm{Pic}}(X)$;
\item[(b)] The underlying complex manifold $({\rm{U}}(E), J)$ admits a family of Hermitian metrics $\Omega(s), s \in (-\infty,\lambda)$, solving the geometric flow
\begin{equation}
\displaystyle \frac{\partial}{\partial s}\Omega(s)= -p(\Omega(s),t), \ \ \ \Omega(0) = \Omega_0,
\end{equation}
$\forall t \in \mathbb{R}$, such that $({\rm{U}}(E),\Omega_{0}) \to (X,\omega_{0})$ is a Hermitian submersion;
\item[(c)] For each $s \in (-\infty , \lambda)$, we have that $({\rm{U}}(E),\Omega(s) , J)$ is a balanced Hermitian manifold, i.e., $\delta \Omega(s) = 0$, $\forall s \in (-\infty , \lambda)$.
\end{enumerate}
 \end{thmx}
Consider the pluriclosed flow
\begin{equation}
\label{pluriclosedflow}
       \begin{cases}
       \displaystyle \frac{\partial }{\partial s} \omega(s)= -\rho_{B}^{1,1}(\omega(s))\\
  \omega(0)=\omega_0   .
   \end{cases} 
\end{equation}
where $\rho_{B}^{1,1}(\omega(s))$ is the $(1,1)$-part of the Bismut Ricci form of $\omega(s)$. In \cite{streets2010parabolic}, Streets and Tian showed that this flow preserves the pluriclosed condition. That means the solution is pluriclosed at any time if the initial data are pluriclosed. More recently, in \cite{streets2016pluriclosed}, Streets proved long time existence and convergence for the pluriclosed flow with arbitrary initial data on certain complex manifolds. \\

In the above setting, as particular consequence of item (b) of the last theorem, we have the following result.
\begin{Coro}
\label{Balanced}
In the setting of Theorem \ref{TeoremaB}, the family of Hermitian metrics $\Omega(s), s \in (-\infty,\lambda)$, on $({\rm{U}}(E), J)$, defines a solution to the pluriclosed flow
\begin{equation}
\displaystyle \frac{\partial }{\partial s} \Omega(s)= -\rho_{B}^{1,1}(\Omega(s)), \ \  \Omega(0)=\Omega_0,
\end{equation}
such that $({\rm{U}}(E),\Omega(s), J)$ is a balanced Hermitian manifold for all $s \in (-\infty , \lambda)$.
\end{Coro}

An interesting feature of the above corollary is that, since $({\rm{U}}(E),\Omega(s),J)$ is non-K\"{a}hler (see for instance \cite{hofer1993remarks}, \cite{poddar2018group}) and balanced, the Hermitian metrics $\Omega(s), s \in (-\infty,\lambda)$, are not pluriclosed, see for instance \cite{popovici2013aeppli}. In \cite{streets2010parabolic}, \cite{tian2011hermitian} and \cite{streets2011regularity}, it is shown that Eq. (\ref{pluriclosedflow}) is a strictly parabolic system under pluriclosed assumption. From this, one gets the short-time existence and some regularity results to solutions. More precisely, in the case that the initial condition of the pluriclosed flow (\ref{pluriclosedflow}) is a pluriclosed metric, it was shown in \cite{streets2010parabolic} that the pluriclosed flow is equivalent to Hermitian curvature flow (HCF)
\begin{equation}
       \begin{cases}
       \displaystyle \frac{\partial }{\partial s} g(s)= -S + Q^{1}\\
  g(0)=g   
   \end{cases} 
\end{equation}
where $S$ is the second Chern–Ricci curvature and $Q^{1}$ is a certain quadratic polynomial in the torsion of the underlying Chern connection. This last fact allows one to obtain the local existence theorem for the pluriclosed flow from the general regularity theorem given in \cite{tian2011hermitian}. In view of these facts, as far as the authors are aware about, none example of balanced non-pluriclosed solution to the pluriclosed flow on non-Kähler manifolds is known in the literature. In this sense, the result presented in Corollary \ref{Balanced} is the first example in the literature of balanced non-pluriclosed solution to the pluriclosed flow on a non-Kähler manifold. \\

Another important question of interest in this work is the convergence of the Chern-Ricci flow. It was shown in \cite{gill2011convergence} that if $(M,J,\omega_0 )$ is a Hermitian manifold with a vanishing Bott-Chern class, then the Chern-Ricci flow admits a long-time solution that converges to a Chern-Ricci flat metric on $(M,J)$. Similarly \cite{tosatti2015evolution} proved that if the first Chern class of $M$ is negative, the flow also admits a long-time solution which, when normalized, converges to a Kähler-Einstein metric. Often, the solution to the flow does not converge to a metric on the same underlying manifold. In such cases, we instead consider convergence in the Gromov-Hausdorff sense. Notably, significant results on this type of convergence for complex surfaces have been established by \cite{tosatti2013chern}, including the cases of Hopf and Inoue surfaces.\\

In the setting of Theorem \ref{TeoremaA} and Theorem \ref{TeoremaB}, considering the underlying principal torus bundle
\begin{equation}
T^{2n} \hookrightarrow {\rm{U}}(E) \to X,
\end{equation}
we apply the results of \cite{correa2024bundle} to compute the collapse of geometric flows of Hermitian structures in the Gromov-Hausdorff sense. In this case, we obtain the following result.
 
\begin{Coro}
\label{GHHermitian}
If the holonomy group ${\rm{Hol}}_{u} (\Theta)$ is closed in $T^{2n}$, then
\begin{equation*}
    \lim_{s \rightarrow \lambda} d_{GH} \big(({\rm{U}}(E) , d_{\Omega_s}), (T^{2n-l}, d_g )\big) =0,
\end{equation*}
where $g$ be the normal metric on the torus $T^{2n-l}$ and $l = \dim{\rm{Hol}}_{u} (\Theta)$.
\end{Coro}

The result above provides an explicit description for the limiting behavior of the geometric flows considered in Theorem \ref{TeoremaA} and Theorem \ref{TeoremaB}. It is worth mentioning that some results related to the convergence of the Chern-Ricci flow and the pluriclosed flow also were obtained in \cite{streets2016pluriclosed}, \cite{barbaro2022global}, \cite{tosatti2015evolution}, \cite{correa2023levi}, \cite{fusi2024pluriclosed}, \cite{liang2024continuity}.\\ 

As an application of the above results, we describe explicitly the Chern-Ricci flow and pluriclosed flow on certain principal $T^{2}$-bundles over the Fano threefold ${\mathbb{P}}(T_{{\mathbb{P}^{2}}})$. In the particular case of the pluriclosed flow, we compute explicitly its Gromov-Hausdorff limit. For more on the dynamics of geometric flows on the Fano threefold ${\mathbb{P}}(T_{{\mathbb{P}^{2}}})$, see \cite{cavenaghi2023complete} and references therein.

\subsection*{Acknowledgments.}E. M. Correa is supported by S\~{a}o Paulo Research Foundation FAPESP grant 2022/10429-3. L. Grama research is partially supported by S\~ao Paulo Research Foundation FAPESP grants 2018/13481-0 and 2023/13131-8 and G. Galindo is supported by the Coordination for the Improvement of Higher Education Personnel CAPES.

\section{Canonical connections}

Let us start by recalling some basic facts about the Ricci forms of the canonical connections on almost Hermitian manifolds. A Riemannian manifold $(M,g)$ is called an almost Hermitian manifold if it is equipped with an almost-complex structure $J$, that is, an orthogonal map $J: TM \rightarrow TM$ that satisfies $J^2 = -Id$. On such manifolds we define the fundamental 2-form as follows
\begin{equation}
    \omega(X,Y) = g(JX,Y).
\end{equation}
From above we refer to $\omega$ as the metric of the almost Hermitian manifold.\\

The almost-complex structure $J: TM \rightarrow TM$ is called integrable (or a complex structure) if the Nijenhuis tensor 
\begin{equation}
    N(X,Y) = [JX,JY]-J[JX,Y]-J[X,JY]-[X,Y]
\end{equation}
vanishes. In that case, we call $(M,g,J)$ a Hermitian manifold. Now we recall some basic concepts which we shall consider throughout this work. 

\begin{definition}
A Hermitian manifold $(M,g,J)$ of complex dimension $n$ is said to be balanced if $\omega$ is co-closed (i.e., $\delta \omega = 0$).
\end{definition}
\begin{remark}
Notice that, since $\ast \omega^{k} = \frac{k!}{(n-k)!}\omega^{n-k}$, it follows that $(M,\omega,J)$ is balanced if, and only if, $d\omega^{n-1} = 0$.
\end{remark}

\begin{definition}
Given a Hermitian manifold $(M,\omega,J)$, we say that the Hermitian metric $\omega$ is pluriclosed if 
\begin{equation}
\partial \overline{\partial}\omega = 0.
\end{equation}
\end{definition}

On almost Hermitian manifolds we can define a family of connections $\nabla^{t}$, $t \in \mathbb{R}$, called the canonical connections \cite{gauduchon1997Hermitian}. For Hermitian manifolds they can be described in terms of the Levi-Civita connection $\nabla^{LC}$ by
\begin{equation}\label{Conections}
    g(\nabla^{t}_X Y, Z) =   g(\nabla^{LC}_X Y, Z) +\frac{t-1}{4} d^c \omega (X,Y,Z) + \frac{t+1}{4} d^c \omega (X,JY,JZ),
\end{equation}
where $d^c \omega (X,Y,Z) = - d\omega(JX,JY,JZ)$. In the above family we have the following distinguished connections:
\begin{enumerate}
    \item [(1)] If $t=1$, $\nabla^{1}$ is the Chern connection; it is characterized by the torsion having no $(1,1)$-component,
    \item[(2)] If $t=0$, $\nabla^{0}$ is the Lichnerowicz connection; it is characterized (among the canonical connections) by the torsion having no $(2,0)$-component,
    \item[(3)] If $t=-1$, $\nabla^{-1}$ is the Strominger-Bismut connection; it is the unique connection such that $T-N$ is skew-symmetric, where $T$ is the torsion.
\end{enumerate}
\begin{remark}

The family of canonical connections $\nabla^t$ can still be defined for almost Hermitian manifolds. However, Equation (\ref{Conections}) changes, for more details, we suggest  \cite{gauduchon1997Hermitian} and \cite{vezzoni2013note}.
\end{remark}
Let $R(\nabla^t )$ be the curvature tensor of $\nabla^t$ and let $p(\omega,t) = \sqrt{-1}{\rm{tr}}(R(\nabla^t ))$ be the associated Ricci form. Consider the $(1,1)$ part of $p(\omega,t)$ given by
\begin{equation}
\label{(1,1)partGauduchon}
    p_1 (\omega,t) = \frac{1}{2}(p(\omega,t) +Jp(\omega,t)),
\end{equation}
where $Jp(\omega,t)$ means
\begin{equation*}
    Jp(\omega,t) (X,Y) = p(\omega,t) (JX,JY).
\end{equation*}
In the above setting, we have the following definition.

\begin{definition}
We say that a Hermitian manifold $(M,\omega,J)$ is $t$-Gauduchon Ricci-flat, for some $t \in \mathbb{R}$, if $p_1 (\omega,t)=0$.
\end{definition}

For the case $t=1$ we have that $p(\omega,1)$, called the Chern-Ricci form, is closed and locally exact, that is, $p(\omega,1)=d\theta$, where $\theta$ is the $1$-form given by
\begin{equation*}
    \theta(X)=-\frac{1}{2}\sum_i g([X,e_i ] ,Je_i ) -g([X,Je_i ], e_i ) +g([JX, e_i ] , e_i ) + g([JX,Je_i ] , e_i ),
\end{equation*}
where $\{e_1 ,...,e_n , Je_1 ,...,Je_n \}$ is a local orthonormal frame for $M$, see for instance \cite{vezzoni2013note}. Moreover, since 
\begin{enumerate}
\item $\rho(\omega,1) = \sqrt{-1}{\rm{tr}}(R(\nabla^{Ch})) = \rho_{1}(\omega,1),$
\item ${\rm{d}} \delta \Omega + J({\rm{d}} \delta \Omega) = 2 ( \partial \partial^{\ast} \Omega + \bar{\partial}\bar{\partial}^{\ast}\Omega),$
\end{enumerate}
we conclude that 
\begin{equation}
\rho_{1}(\omega,t) =  \rho_{1}(\omega,1)+ \frac{t-1}{2}\big ( \partial \partial^{\ast} \omega + \bar{\partial}\bar{\partial}^{\ast}\omega\big).
\end{equation}
In particular, if $t = -1$, we obtain the $(1,1)$-part of the Bismut–Ricci form $\rho_{B}(\omega) = \rho(\omega,-1)$, i.e., we have 
\begin{equation}
\rho_{1}(\omega,-1) =  \rho_{1}(\omega,1)-\big ( \partial \partial^{\ast} \omega + \bar{\partial}\bar{\partial}^{\ast}\omega\big).
\end{equation}
For the sake of simplicity, we shall denote $\rho_{B}^{1,1}(\omega)  := \rho_{1}(\omega,-1)$.

\section{Geometric Flows on Hermitian manifolds}

In what follows, we briefly present the geometric flows of Hermitian metrics which we are concerned in this work, for more details on the subject, we suggest \cite{tian2011hermitian}, \cite{tosatti2015evolution}, and \cite{streets2010parabolic}.\\

\par A Hermitian manifold $(M,g,J)$ is called Kähler if $d\omega=0$. In this case, all the canonical connections coincide with the Levi-Civita one, and we denote its Ricci form only by $p_{\omega}$. On Kähler manifolds, a solution to the Kähler-Ricci flow starting at $\omega_0 = g_0 (J\cdot,\cdot)$ is a family of K\"{a}hler metrics $\omega_s = \omega(s)$, which satisfies
\begin{align}
   \begin{cases}
       \displaystyle \frac{\partial}{\partial s}\omega(s) = -p_{\omega_s}\\
  \omega(0)=\omega_0   .
   \end{cases} 
\end{align}
Since the Ricci form is $J$-invariant, this flow coincides with the Ricci flow of the underlying Riemannian metrics.
\begin{definition}
A Hermitian metric $\omega$ on a complex manifold $(M,J)$ is said to be Kähler-Einstein metric if it is Kähler and satisfies the equation
    \begin{equation*}
        p_\omega =\lambda \omega
    \end{equation*}
for some $\lambda \in \mathbb{R}$, that is, the Ricci form is a scalar multiple of the metric. 
\end{definition}

\begin{remark}
Given a Kähler-Einstein manifold $(M,\omega,J)$, with positive scalar curvature, we shall assume that $\omega = \lambda p_\omega$, inverting the side of the constant.
\end{remark}

If $(M,\omega_{0},J)$ is a Kähler-Einstein manifold with positive scalar curvature, such that $\omega_0 =  \lambda p_{\omega_0}$, then a solution to the Kähler-Ricci flow starting at $\omega_0$ is  
\begin{equation}
    \omega_s =(\lambda-s)p_{\omega_0} = (\lambda-s) p_{\omega_s}, 
\end{equation}
such that $s \in (-\infty , \lambda)$. \\

For non-Kähler manifolds, the Ricci flow does not necessarily  preserve the Hermitian condition $g(J\cdot,J\cdot) = g$. Thus, instead of evolving the metric by the Ricci form, one can evolve the metric by its Chern-Ricci form, i.e., by means of the Chern-Ricci flow, see \cite{tosatti2022chern}, \cite{zheng2018almost}. In this setting, given a Hermitian manifold, $(M,\omega_{0},J)$, a solution to the Chern-Ricci flow starting at a metric $\omega_0$ is a one-parameter family of metrics $\omega_s = \omega(s)$ satisfying the PDE
\begin{equation}
       \begin{cases}
       \displaystyle \frac{\partial }{\partial s} \omega(s)= -p_1 (\omega(s),1)\\
  \omega(0)=\omega_0   .
   \end{cases} 
\end{equation}

\begin{remark}
In the above setting, if the starting metric $\omega_{0}$ is Kähler, then the Chern-Ricci flow coincides with the Kähler-Ricci flow.
\end{remark}

Another important flow of Hermitian metrics, introduced in \cite{streets2010parabolic} by Streets and Tian, is the initial value problem
\begin{equation}
       \begin{cases}
       \displaystyle \frac{\partial }{\partial s} \omega(s)= -\rho_{B}^{1,1}(\omega(s))\\
  \omega(0)=\omega_0   .
   \end{cases} 
\end{equation}
Since pluriclosed condition are preserved along the above flow, it is called pluriclosed flow. As mentioned in \cite{streets2013regularity}, equation above falls into a general class of flows of Hermitian metrics, and as shown by Streets and Tian in \cite{tian2011hermitian}, solutions to the above initial value problem exist as long as the Chern curvature, torsion and covariant derivative of torsion are bounded.

\section{Generalities on Flag manifolds}
\par In this section, we give a brief account on flag manifolds, for more details see \cite{arvanitoyeorgos2006geometry} or \cite{alekseevsky1997flag}.
\begin{definition}
 A complex flag manifold is a complex homogeneous space $G^{\mathbb{C}}/ P$, where $G^{\mathbb{C}}$ is a connected, simply connected and semisimple complex Lie group and $P \subset G^{\mathbb{C}}$ is a parabolic subgroup, i.e., a Lie subgroup which contains a Borel subgroup.
\end{definition}

\begin{remark}
Fixed a compact real form $G \subset G^{\mathbb{C}}$, in this work, unless otherwise stated, we shall assume that every Hermitian metric on a complex flag manifold is $G$-invariant. 
\end{remark}

In terms of Lie algebras, let $\mathfrak{g}^{\mathbb{C}} = \mathfrak{t}^\mathbb{C} \oplus \sum_{\alpha \in R} \mathfrak{g}_\alpha $ be the root space decomposition of $\mathfrak{g}^{\mathbb{C}}$, where $\mathfrak{t}^\mathbb{C}$ is a Cartan subalgebra and $R$ is the associated set of roots. Fixed a system of simple roots $\Pi \subset R$, let $R^+$ be the underlying set of positive roots. From this data, we can construct a Borel subalgebra $\mathfrak{b}$ by setting
\begin{equation}
\label{canonicalborel}
    \mathfrak{b} = \mathfrak{t}^\mathbb{C} \oplus \sum_{\alpha \in R^+}\mathfrak{g}_\alpha.
\end{equation}
Consider now the followng result.
\begin{theorem}
    Any two Borel subgroups are conjugated to each other.
\end{theorem}

From the above result, we can always assume that the Lie algebra of a Borel subgroup has the form given in Equation (\ref{canonicalborel}).\\

Considering $B = \exp(\mathfrak{b})$, every parabolic subgroup $P$ (up to conjugation) can be described in terms of parabolic Lie subalgebra as follows: Given a subset $\Pi_P \subset \Pi$ define 
\begin{center}
$R_P = R \cap {\rm{Span}} [ \Pi_P ],$ 
\end{center}
and set
 \begin{equation}
     \mathfrak{p} = \mathfrak{t}^\mathbb{C} \oplus \sum_{\alpha \in R^+}\mathfrak{g}_\alpha \oplus \sum_{\alpha \in R_P / R^+} \mathfrak{g}_{\alpha} =  \mathfrak{b} \oplus \sum_{\alpha \in R_P / R^+} \mathfrak{g}_{\alpha}.
 \end{equation}
From above we obtain a parabolic Lie subalgebra $\mathfrak{p} \subset \mathfrak{g}^{\mathbb{C}}$ that defines a parabolic subgroup $P \subset G^\mathbb{C}$. A consequence of the above construction is that every flag manifold is completely determined by the choice of a subset $\Pi_P \subset \Pi$.



\begin{remark}
As a quotient of complex Lie group, every flag manifold $X= G^\mathbb{C}/P$ admits complex structures, and each complex structure gives rise to a homogeneous Kähler-Einstein metric with positive scalar curvature \cite{matsushima1972remarks}.
\end{remark}

\par Given a complex manifold $X$, we denote by ${\rm{Pic}}(X)$ the Picard group of $X$, i.e., the group of isomorphism classes of holomorphic line bundles over $X$ with the group operation given by the tensor product. \\

When $X=G^\mathbb{C}/P$ is a complex flag manifold, there is a simple description for ${\rm{Pic}}(X)$ in terms of the fundamental weights of the Lie algebra $\mathfrak{g}^{\mathbb{C}}$. For every root $\alpha \in R$, consider
\begin{equation*}
    \alpha^{\vee} :=\frac{2}{\langle \alpha , \alpha \rangle}\alpha ,
\end{equation*}
where $\langle -,-\rangle$  is the associated Killing form. From above, the fundamental weights $\{\varpi_\alpha \ | \ \alpha \in \Pi \}  \subset \mathfrak{t^{\mathbb{C}}}^*$ are defined by requiring that $\langle \varpi_\alpha , \beta^{\vee} \rangle =0$   if $\alpha \neq \beta$ and  $\langle \varpi_\alpha , \alpha^{\vee} \rangle =1$, for every $\alpha, \beta \in \Pi$. The set of integral dominant weights of $\mathfrak{g}^{\mathbb{C}}$ is defined by
\begin{equation}
    \Lambda^{+} := \bigoplus_{\alpha \in \Pi} \mathbb{Z}_{\geq 0}\varpi_\alpha.
\end{equation}
Given a complex flag manifold $X=G^\mathbb{C}/P$, considering $\Pi_P \subset \Pi$ determined by $P$, we call
\begin{equation}
       \Lambda_{P} := \bigoplus_{\alpha \in \Pi/\Pi_P} \mathbb{Z}\varpi_\alpha
\end{equation}
the weights of $X$. In this last case, we have an isomorphism
\begin{equation*}
\label{isopic}
    {\rm{Pic}}(X) \cong \Lambda_{P}, \ \ E \mapsto \lambda(E),
\end{equation*}
For details on how this isomorphism works see \cite{snow1989homogeneous} or \cite{correa2023t}.
\begin{remark}
In the above setting, we shall denote by $\mathscr{O}_{\alpha}(1) \in {\rm{Pic}}(X)$ the generators satisfying $\lambda(\mathscr{O}_{\alpha}(1)) = \varpi_{\alpha}$, $\forall \alpha \in \Pi/\Pi_P$. Also, we will denote 
\begin{center}
$\mathscr{O}_{\alpha}(s) := \underbrace{\mathscr{O}_{\alpha}(1) \otimes \cdots \otimes  \mathscr{O}_{\alpha}(1)}_{s-{\text{times}}}$, 
\end{center}
for every $s \in \mathbb{Z}$ and every $\alpha \in \Pi/\Pi_P$.
\end{remark}
\begin{remark}
We note that the Fano index of a complex flag manifold $X$, defined by the choice of some $\Pi_P \subset \Pi$, is the greatest common divisor 
\begin{equation*}
    \rm{I}(X) = GCD \{ \langle  \delta_{P}, \alpha^{\vee}\rangle | \alpha \in \Pi/\Pi_P \},
\end{equation*}
where $\delta_P = \sum_{\alpha \in R^+ /R_{P}}\alpha$. For more details, see \cite{correa2019homogeneous}.\\
\end{remark}
Given a Kähler metric $\omega$ on $X$ and $E \in {\rm{Pic}}(X)$, we define the degree of $E$, relative to $\omega$, as\footnote{Sometimes it is convenient to denote $\deg_{\omega}(E)$ to emphasize the underlying K\"{a}hler metric.}
\begin{equation}
    \deg(E) := \int_X c_1 (E) \wedge \omega^{n-1},
\end{equation}
where $c_1 (E)$ is the first Chern class of $E$. In this work, we shall consider the subgroup of ${\rm{Pic}}(X)$ defined by
\begin{equation}
{\rm{Pic}}^{0}_{\omega}(X) = \big \{ E \in {\rm{Pic}}(X)  \ \big | \ \deg(E) = 0\},
\end{equation}
for some fixed K\"{a}hler metric $\omega$ on $X$. A complete description of ${\rm{Pic}}^{0}_{\omega}(X)$ can be found in \cite{correa2023t}.

In what follows, we present a detailed example which summarizes the above ideas. This example will be considered to illustrate our main results later on.

\begin{example}
\label{exampleprojtangent}
Consider the complex Lie group $G^{\mathbb{C}} = {\rm{SL}}_{3}(\mathbb{C})$. In this case, the structure of the associated Lie algebra $\mathfrak{sl}_{3}(\mathbb{C})$ can be completely determined by means of its Dynkin diagram
\begin{center}
${\dynkin[labels={\alpha_{1},\alpha_{2}},scale=3]A{oo}} $
\end{center}
Choosing the Cartan subalgebra $\mathfrak{h} \subset \mathfrak{sl}_{3}(\mathbb{C})$ of diagonal matrices, we have the associated simple root system given by $\Pi  = \{\alpha_{1},\alpha_{2}\}$, such that 
\begin{center}
$\alpha_{j}({\rm{diag}}(d_{1},d_{2},d_{3})) = d_{j} - d_{j+1}$, $j = 1,2$,
\end{center}
$\forall {\rm{diag}}(d_{1},d_{2},d_{3}) \in \mathfrak{h}$. The set of positive roots in this case is given by 
\begin{center}
$R^{+} = \{\alpha_{1}, \alpha_{2}, \alpha_{3} = \alpha_{1} + \alpha_{2}\}$. 
\end{center}
Considering the Cartan-Killing form\footnote{We have $\kappa(X,Y) = 6{\rm{tr}}(XY), \forall X,Y \in \mathfrak{sl}_{3}(\mathbb{C})$, see for instance \cite[Chapter 10, \S 4]{procesi2007lie}.} $\kappa(X,Y) = {\rm{tr}}({\rm{ad}}(X){\rm{ad}}(Y)), \forall X,Y \in \mathfrak{sl}_{3}(\mathbb{C})$, it follows that $\alpha_{j} = \kappa(\cdot,h_{\alpha_{j}})$, $j =1,2,3$, such that\footnote{Notice that $\langle \alpha_{j},\alpha_{j} \rangle = \alpha_{j}(h_{\alpha_{j}}) = \frac{1}{3}, \forall j = 1,2,3.$} 
\begin{equation}
h_{\alpha_{1}} =\frac{1}{6}(E_{11} - E_{22}), \ \ h_{\alpha_{2}} =\frac{1}{6}(E_{22} - E_{33}), \ \ h_{\alpha_{3}} =\frac{1}{6}(E_{11} - E_{33}),
\end{equation}
here we consider the matrices $E_{ij}$ as being the elements of the standard basis of ${\mathfrak{gl}}_{3}(\mathbb{C})$. Moreover, we have the following relation between simple roots and fundamental weights:
\begin{center}
$\displaystyle{\begin{pmatrix}
\alpha_{1} \\ 
\alpha_{2}
\end{pmatrix} = \begin{pmatrix} \ \ 2 & -1 \\
-1 & \ \ 2\end{pmatrix} \begin{pmatrix}
\varpi_{\alpha_{1}} \\ 
\varpi_{\alpha_{2}}
\end{pmatrix}, \ \ \ \begin{pmatrix}
\varpi_{\alpha_{1}} \\ 
\varpi_{\alpha_{2}}
\end{pmatrix} = \frac{1}{3}\begin{pmatrix} 2 & 1 \\
1 & 2\end{pmatrix} \begin{pmatrix}
\alpha_{1} \\ 
\alpha_{2}
\end{pmatrix}},$
\end{center}
here we consider the Cartan matrix $C = (C_{ij})$ of $\mathfrak{sl}_{3}(\mathbb{C})$ given by 
\begin{equation}
\label{Cartanmatrix}
C = \begin{pmatrix}
 \ \ 2 & -1 \\
-1 & \ \ 2 
\end{pmatrix}, \ \ C_{ij} = \frac{2\langle \alpha_{i}, \alpha_{j} \rangle}{\langle \alpha_{j}, \alpha_{j} \rangle},
\end{equation}
for more details on the above subject, see for instance \cite{Humphreys}. Fixed the standard Borel subgroup $B \subset {\rm{SL}}_{3}(\mathbb{C})$, i.e.,
\begin{center}
$B = \Bigg \{ \begin{pmatrix} \ast & \ast & \ast \\
0 & \ast & \ast \\
0 & 0 & \ast \end{pmatrix} \in {\rm{SL}}_{3}(\mathbb{C})\Bigg\},$
\end{center}
taking $I = \emptyset$, we obtain the flag manifold defined by the  Wallach space 
\begin{equation}
{\mathbb{P}}(T_{{\mathbb{P}^{2}}}) = {\rm{SL}}_{3}(\mathbb{C})/B. 
\end{equation}
In this particular case, we have the following facts:
\begin{enumerate}
\item[(i)] $H^{2}({\mathbb{P}}(T_{{\mathbb{P}^{2}}}),\mathbb{R}) = H^{1,1}({\mathbb{P}}(T_{{\mathbb{P}^{2}}}),\mathbb{R}) = \mathbb{R}[{\bf{\Omega}}_{\alpha_{1}}] \oplus \mathbb{R}[{\bf{\Omega}}_{\alpha_{2}}]$;
\item[(ii)] ${\rm{Pic}}({\mathbb{P}}(T_{{\mathbb{P}^{2}}})) = \Big \{ \mathscr{O}_{\alpha_{1}}(s_{1}) \otimes \mathscr{O}_{\alpha_{1}}(s_{1}) \ \Big | \ s_{1}, s_{2} \in \mathbb{Z}\Big \}$.
\end{enumerate}
Here we denote by 
\begin{center}
${\bf{\Omega}}_{\alpha_{1}} \in c_{1}(\mathscr{O}_{\alpha_{1}}(1))$ \ \ \ and \ \ \ ${\bf{\Omega}}_{\alpha_{2}} \in c_{1}(\mathscr{O}_{\alpha_{1}}(1))$, 
\end{center}
the unique ${\rm{SU}}(3)$-invariant representatives inside each cohomology class. Let $\vartheta_{0}$ be the unique ${\rm{SU}}(3)$-invariant K\"{a}hler metric on ${\mathbb{P}}(T_{{\mathbb{P}^{2}}})$, such that 
\begin{center}
$[\vartheta_{0}] = c_{1}({\mathbb{P}}(T_{{\mathbb{P}^{2}}}))$. 
\end{center}
Since $\lambda({\bf{K}}_{{\mathbb{P}}(T_{{\mathbb{P}^{2}}})}^{-1}) = \delta_{B} = 2(\varpi_{\alpha_{1}} + \varpi_{\alpha_{2}})$ (see Eq. (\ref{isopic})), it follows that 
\begin{equation}
\vartheta_{0} = \langle \delta_{B}, \alpha_{1}^{\vee} \rangle {\bf{\Omega}}_{\alpha_{1}} + \langle \delta_{B}, \alpha_{2}^{\vee} \rangle {\bf{\Omega}}_{\alpha_{2}} = 2 \big ({\bf{\Omega}}_{\alpha_{1}} + {\bf{\Omega}}_{\alpha_{2}}\big),
\end{equation}
in particular, notice that $\lambda([\vartheta_{0}]) = \delta_{B}$. Now, let us describe the subgroup
\begin{equation}
{\rm{Pic}}^{0}_{\vartheta_{0}}({\mathbb{P}}(T_{{\mathbb{P}^{2}}})) = \Big \{ E \in {\rm{Pic}}({\mathbb{P}}(T_{{\mathbb{P}^{2}}})) \ \Big | \ \deg_{\vartheta_{0}}(E) = 0\Big \}.
\end{equation}
Given $E = \mathscr{O}_{\alpha_{1}}(s_{1}) \otimes \mathscr{O}_{\alpha_{1}}(s_{1})$, we have 
\begin{equation}
\deg_{\vartheta_{0}}(E) = 0 \iff s_{1}\int_{{\mathbb{P}}(T_{{\mathbb{P}^{2}}})}{\bf{\Omega}}_{\alpha_{1}} \wedge \vartheta_{0}^{2} + s_{2}\int_{{\mathbb{P}}(T_{{\mathbb{P}^{2}}})}{\bf{\Omega}}_{\alpha_{2}} \wedge \vartheta_{0}^{2} = 0.
\end{equation}
Since $\Lambda_{\vartheta_{0}}({\bf{\Omega}}_{\alpha_{1}})$ and $\Lambda_{\vartheta_{0}}({\bf{\Omega}}_{\alpha_{2}})$ are constants, it follows that  
\begin{equation}
\deg_{\vartheta_{0}}(E) = 0 \iff s_{1}\Lambda_{\vartheta_{0}}({\bf{\Omega}}_{\alpha_{1}}) + s_{2}\Lambda_{\vartheta_{0}}({\bf{\Omega}}_{\alpha_{2}}) = 0.
\end{equation}
The coefficients of the above linear equation can be computed by using relations 
\begin{equation}
\langle \lambda,\alpha_{1}^{\vee} \rangle = a, \ \ \langle \lambda,\alpha_{2}^{\vee} \rangle = b, \ \ \langle \lambda,\alpha_{3}^{\vee} \rangle = a + b,
\end{equation}
for every $\lambda = a\varpi_{\alpha_{1}} + b \varpi_{\alpha_{2}}$, $a,b \in \mathbb{Z}$, see for instance \cite[p. 14-15]{correa2023deformed}. Therefore, we have 
\begin{equation}
\deg_{\vartheta_{0}}(E) = 0 \iff \frac{3}{4}(s_{1} + s_{2}) = 0.
\end{equation}
From above, we conclude that ${\rm{Pic}}^{0}_{\vartheta_{0}}({\mathbb{P}}(T_{{\mathbb{P}^{2}}})) = \big \langle \mathscr{O}_{\alpha_{1}}(-1)\otimes \mathscr{O}_{\alpha_{2}}(1) \big \rangle $.
\end{example}

\section{Proof of main results}

In what follows, we will summarize the main ideas introduced in \cite{grantcharov2008calabi} and \cite{correa2023t} to construct $t$-Gauduchon Ricci-flat metrics on principal torus bundles over flag manifolds. Then, once we have presented the main ingredients, we will prove our main results\\

Let $P$ be a  principal $T^{2n}$-bundle over a compact Hermitian manifold $(M,g_0 ,J_M)$. By choosing a principal connection $\Theta \in \Omega^1 (M,Lie(T^{2n}))$, given by
\begin{equation}
    \Theta= 
    \begin{bmatrix}
    \sqrt{-1} \Theta_1 &\hdots & 0\\
    \vdots & \ddots &\vdots\\
    0 &\hdots &\sqrt{-1}\Theta_{2n}
    \end{bmatrix},
\end{equation}
such that $d\Theta_j = \pi^* (\psi_j )$, suppose that $\psi_j \in \Omega^{1,1}(M)$.  From this, we can construct a complex structure $J$ on $P$ by using the horizontal lift of the complex structure $J_{M}$ defined on the base to $\ker(\Theta)$ (horizontal space) and, since the vertical space is identified with the tangent space of an even-dimensional torus, we can set on the vertical direction
\begin{equation}\label{Complexstructure}
    \begin{cases}
J\Theta_{2k-1} = -\Theta_{2k}\\
        J\Theta_{2k} = \Theta_{2k-1}
    \end{cases}
\end{equation}
for every $1 \leq k \leq n$. So we have a well-defined almost complex structure $J$ on $P$. A straightforward computation shows that it is integrable (\cite{grantcharov2008calabi}, Lemma 1).\\
\par Considering the Hermitian metric $g_0$ on $M$ we can use the connection $\Theta$ to construct a Hermitian metric on $(P,J)$. In fact, we can define 
\begin{equation}
    g_P = \pi^* (g_0 ) +\frac{1}{2}{\rm{tr}}(\Theta \odot \Theta),
\end{equation}
where $a \odot b =a \otimes b + b \otimes a$, $\forall a,b \in \Omega^1 (P)$. It follows from Equation (\ref{Complexstructure}) that the fundamental 2-form $\Omega$ of $(P,g_P , J)$ is
\begin{equation}
    \Omega = \pi^* (\omega_0 ) +\frac{1}{2} {\rm{tr}}(\Theta \wedge J\Theta),
\end{equation}
where we consider $a \wedge b = a \otimes b - b \otimes a$, $\forall a,b \in \Omega^1 (P)$ and $\omega_0$ is the fundamental $2$-form of $(M,g_0 , J_M )$.\\

From the above construction, we have the following description for the Ricci forms of the canonical connections $\nabla^t$ on the Hermitian manifold $(P,\Omega , J)$ 
\begin{equation}\label{tricciform}
    p(\Omega , t) = \pi^* \Big( p(\omega_0 , 1)  + \frac{t-1}{2}\sum_{j=1}^{2n} \Lambda_{\omega_0}(\psi_j ) \psi_j \Big),
\end{equation}
where $p(\omega_0 ,1)$ is the Ricci form of the Chern connection associated with $(M, g_0,J)$ and $\Lambda_{\omega_0}$ is the dual of the Lefschetz operator $L_{\omega_0} = \omega_0 \wedge (-)$, form more details, see (\cite{grantcharov2008calabi}, Proposition 5).\\

Also, we can check that 
\begin{equation}\label{balanced}
    \delta \Omega =\pi^* (\delta \omega_0 ) + \sum_{j=1}^{2n} \pi^* (\Lambda_{\omega_0}(\psi_j ))\Theta_j. 
\end{equation}
It is worth to point out that, from Equation \ref{tricciform}, it follows that the Ricci forms of the canonical connections $\nabla^{t}$ on $(P,\Omega,J)$ are all of $(1,1)$-type, so
\begin{equation*}
      p(\Omega , t)=  p_1 (\Omega , t).
\end{equation*}
Therefore, the condition of being $t$-Gauduchon Ricci-flat for $(P, \Omega , J)$ is equivalent to
\begin{equation}
    p(\omega_0 ,1  )= -\frac{t-1}{2}\sum_{j=1}^{2n} \Lambda_{\omega_0}(\psi_j ) \psi_j .
\end{equation}
In the particular case that the base manifold $(M,g_{0},J_{M})$ is a complex flag manifold $X$, with Picard number $\varrho(X)>1$\footnote{This condition is related to the existence of primitive $(1,1)$-classes on $X$, for more details, see \cite{correa2023t}.}, and $P$ is a suitable principal torus bundle over it, we can solve the $t$-Gauduchon Ricci-flat equation proceeding as follows 
\\

Fix a $G$-invariant K\"{a}hler metric $\omega_{0}$ on $X$. Let $\mathcal{O}_{X}(1)$ be the irreducible root of the anti-canonical line bundle over $X$
\begin{equation}
    \mathcal{O}_{X}(1)= \frac{1}{{\rm{I}}(X)}K^{-1}_X
\end{equation}
where ${\rm{I}}(X)$ is the Fano index of $X$, see (\cite{kollar2013rational}, Chapter 5) for details. Also denote
\begin{equation}
\mathcal{O}_{X}(k)=\mathcal{O}_{X}(1)^{\otimes k}.
\end{equation}
Given $F_1 ,..., F_{2n-1} \in  \rm{Pic}_{\omega_0 }^{0}(X )$ and $k \in \mathbb{Z}^*$, we define
\begin{equation} \label{defini-k}
    E= \mathcal{O}_{X}(k) \oplus F_1 \oplus ... \oplus F_{2n-1}.
\end{equation}
Taking a Hermitian structure $h$ on $E$, defined by
\begin{equation}
    h= h_0 \oplus...\oplus h_{2n-1},
\end{equation}
where $h_0$ is a Hermitian structure on $\mathcal{O}_{X}(k)$ and $h_i$ is  a Hermitian structure on $F_i$, for $1 \leq i \leq 2n-1$. We have that the associated unitary frame bundle ${\rm{U}}(E)$ is a principal $T^{2n}$-bundle over $X$.\\

If the flag manifold $X$ has Picard number $\varrho(X)>1$, then there is a principal connection $\Theta \in \Omega^1 ({\rm{U}}(E),Lie(T^{2n}))$
\begin{equation*}
     \Theta= 
    \begin{bmatrix}
    \sqrt{-1} \Theta_1 &\hdots & 0\\
    \vdots & \ddots &\vdots\\
    0 &\hdots &\sqrt{-1}\Theta_{2n}
    \end{bmatrix}
\end{equation*}
satisfying the following (\cite{correa2023t}, Section 5.1):
\begin{enumerate}
    \item [(1)] $d\Theta_i = \pi^* (\psi_i )$ such that $\frac{\psi_1}{2\pi} \in c_1 (\mathcal{O}_{X}(k)) $, $\frac{\psi_2}{2\pi} \in c_1 ( F_1 )$,..., $\frac{\psi_{2n}}{2\pi} \in c_1 ( F_{2n-1} )$, 
    \item[(2)] $\psi_j$ are $G$-invariant $(1,1)$-forms,
    \item[(3)] $\Lambda_{\omega_{0}} (\psi_j ) = 0$, $\forall j = 2,...,2n$.
      
\end{enumerate}
Using this connection, we equip the principal $T^{2n}$-bundle ${\rm{U}}(E)$ with a complex structure $J$ and Hermitian metric $\Omega$ as described previously.\\

From above, it follows from Equation (\ref{tricciform}) that the Ricci form associated with a canonical connection $\nabla^t$ on $({\rm{U}}(E),\Omega,J)$, for some $t \in \mathbb{R}$, is given by
\begin{equation*}
    p (\Omega,t) = \pi^* \Big(p (\omega_0 ,1) + \frac{t-1}{2}\sum_{i=1}^{2n}\Lambda_{\omega_0} (\psi_i )\psi_i )\Big) =  \pi^* \Big(p_{\omega_0 }+ \frac{t-1}{2}\Lambda_{\omega_0} (\psi_1 )\psi_1 )\Big) 
\end{equation*}
Hence, if we choose $\omega_0$ to be a $G$-invariant Kähler-Einstein metric $\omega_0 = \lambda p_{\omega_0}$, where $p_{\omega_0}$ is the K\"{a}hler-Ricci form of $X$, since $\frac{\psi_1}{2\pi} \in c_1 (\mathcal{O}(k))$ is $G$-invariant, where $k$ is defined as in Equation \ref{defini-k}, it follows that
\begin{equation}
    \psi_1 = \frac{k}{{\rm{I}}(X)}p_{\omega_0} = \frac{k}{\lambda {\rm{I}}(X)}\omega_0. 
\end{equation}
Therefore, we obtain
\begin{equation}
    \Lambda_{\omega_0}(\psi_1 ) = 
    \frac{k}{\lambda {\rm{I}}(X)} \Lambda_{\omega_0}(\omega_0 ) =
    \frac{k \dim_{\mathbb{C}}(X )}{\lambda {\rm{I}}(X)}.
\end{equation}
Thus, the Ricci form satisfies
\begin{equation}
      p (\Omega,t) =\pi^* \Big(\frac{1}{\lambda} \omega_0 +\frac{t-1}{2}\frac{k^2 \dim_{\mathbb{C}}(X )}{\lambda^2 {\rm{I}}(X)^2} \omega_0 \Big) =  \Big(\frac{1}{\lambda} +\frac{t-1}{2}\frac{k^2 \dim_{\mathbb{C}}(X )}{\lambda^2 {\rm{I}}(X)^2} \Big) \pi^*(\omega_{0}),
\end{equation}
and we conclude that
\begin{equation}
    p (\Omega,t)=0 \iff \lambda = \frac{1-t}{2}\frac{k^2 \dim_{\mathbb{C}}(X )}{ {\rm{I}}(X)^2}.
\end{equation}
From above, since $\lambda > 0$, we can solve the $t$-Gauduchon Ricci-flat equation for every $t<1$. Moreover, if we let the metric $\omega_0$ on $X$ evolving through the Kähler-Ricci flow $\omega(s) = (\lambda-s )p_{\omega_0}$, this yields a family of Hermitian metrics
\begin{equation}
   \Omega_s =\pi^* \big ((\lambda-s )p_{\omega_0}\big) +\frac{1}{2} {\rm{tr}}(\Theta \wedge J\Theta ), \ \ \ s \in (-\infty,\lambda),    
\end{equation}
on $(U(E),J)$ that turns out to be the solution to the Chern-Ricci flow. By keeping the above notation, we have the following theorem.

\begin{theorem}\label{Teorema2}
Let $X $ be a complex flag manifold with Picard number $\varrho(X)>1$ and equipped with a Kähler-Einstein metric $\omega_0 =\lambda p_{\omega_0}$. Then, there exists a Hermitian holomorphic vector bundle $(E,h) \to X$, such that:
\begin{enumerate}
\item[(a)] $E= \mathcal{O}_{X}(k) \oplus F_1 \oplus ... \oplus F_{2n-1}$, where $\mathcal{O}_{X}(k), F_1, \ldots, F_{2n-1} \in {\rm{Pic}}(X)$;
\item[(b)] The underlying complex manifold $({\rm{U}}(E), J)$ admits a family of Hermitian metrics $\Omega(s), s \in (-\infty,\lambda)$, solving the Chern-Ricci flow
\begin{equation}
\displaystyle \frac{\partial}{\partial s}\Omega(s)= -p(\Omega(s),1), \ \ \ \Omega(0) = \Omega_0,
\end{equation}
such that $({\rm{U}}(E),\Omega_{0}) \to (X,\omega_{0})$ is a Hermitian submersion;
\item[(c)] For each $s \in (-\infty , \lambda)$, we have that $({\rm{U}}(E),\Omega_s , J)$ is $t$-Gauduchon Ricci-flat if, and only if,
\begin{equation*}
  t =1 - \frac{2(\lambda-s) {\rm{I}}(X)^2}{k^2 \dim_{\mathbb{C}}(X )},
\end{equation*}
where ${\rm{I}}(X)$ is the Fano index of $X$;
\item[(d)] ${\rm{U}}(E)$ does not admit any Kähler metric and, if $I(X)\lambda \in k\mathbb{Z}$, then 
\begin{equation}
c_{1}(T{\rm{U}}(E)) = 0 \in H^{2}({\rm{U}}(E),\mathbb{Z}), 
\end{equation}
i.e., ${\rm{U}}(E)$ is non-Kähler Calabi-Yau.
\end{enumerate}
\end{theorem}

\begin{proof} 
Given $F_1 ,..., F_{2n-1} \in  \rm{Pic}_{\omega_0 }^{0}(X )$ and $k \in \mathbb{Z}^*$, we define
\begin{equation} 
    E= \mathcal{O}_{X}(k) \oplus F_1 \oplus ... \oplus F_{2n-1}.
\end{equation}
From the construction previously described, we have a one-parameter family of Hermitian metrics on $\Omega_{s}$ on $({\rm{U}}(E), J)$, such that 
\begin{equation*}
   \Omega_s =\pi^* \big ((\lambda-s )p_{\omega_0}\big) +\frac{1}{2} {\rm{tr}}(\Theta \wedge J\Theta ).
\end{equation*}
By construction, we have 
\begin{equation*}
    \frac{\partial}{\partial s} \Omega_s = \frac{\partial}{\partial s} \Big( \pi^* (\lambda-s )p_{\omega_0} +\frac{1}{2} {\rm{tr}}(\Theta \wedge J\Theta ) \Big) = -\pi^* p_{\omega_0} = -p(\Omega_s , 1),
\end{equation*}
where the last equality follows from Equation (\ref{tricciform}). Thus, $\Omega_{s}$ is a solution to the Chern-Ricci flow starting at $({\rm{U}}(E),J,\Omega_0 )$. Thus, we obtain items (a) and (b). For item (c), observing that $\rm{Pic}_{\omega_0 }^{0}(X ) = \rm{Pic}_{c\omega_0 }^{0}(X )$ and
 \begin{equation}
\Lambda_{c \omega_{0}}(-) = \frac{1}{c}\Lambda_{\omega_{0}}(-),
 \end{equation}
for every $c \in \mathbb{R}^+$, it follows from the previous construction that 
\begin{equation}
 p (\Omega_{s},t)=0 \iff \lambda - s = \frac{1-t}{2}\frac{k^2 \dim_{\mathbb{C}}(X )}{ {\rm{I}}(X)^2} \iff t =1 - \frac{2(\lambda-s) {\rm{I}}(X)^2}{k^2 \dim_{\mathbb{C}}(X )}.
\end{equation}
We observe now that ${\rm{U}}(E)$ do not carry a K\"{a}hler metric for purely topological reasons, see for instance \cite{hofer1993remarks} and \cite{poddar2018group}. In particular, from Eq. (\ref{tricciform}), it follows that 
\begin{equation}
p(\Omega_{0},1) = \lambda \pi^{\ast}(p_{\omega_{0}}) =  \frac{\lambda{\rm{I}}(X)}{k} \pi^{\ast}(\psi_{1}) = \frac{\lambda{\rm{I}}(X)}{k}{\rm{d}}\Theta_{1}.
\end{equation}
Thus, we conclude that $c_{1}(T{\rm{U}}(E)) = [\frac{p(\Omega_{0},1)}{2\pi}] = 0 \in H^{2}({\rm{U}}(E),\mathbb{R})$. From this, if $\lambda I(X) \in k\mathbb{Z}$, then $c_{1}(T{\rm{U}}(E)) = [\frac{p(\Omega_{0},1)}{2\pi}] = 0 \in H^{2}({\rm{U}}(E),\mathbb{Z})$
\end{proof}

From item $(c)$ of the last theorem, we have the following consequence.

\begin{coro}
The Chern-Ricci flow does not preserve the property of being $t$-Gauduchon Ricci-flat.
\end{coro}

Now we slightly change the construction used in Theorem \ref{Teorema2} to construct an example of the Chern-Ricci flow that does preserve the property of the metric been balanced. We have the following result:

\begin{theorem}
\label{TeoremaBa}
Let $X $ be a complex flag manifold with Picard number $\varrho(X)>1$ and equipped with a Kähler-Einstein metric $\omega_0 =\lambda p_{\omega_0}$. Then, there exists a Hermitian holomorphic vector bundle $(E,h) \to X$, such that:
\begin{enumerate}
\item[(a)] $E=  F_1 \oplus ... \oplus F_{2n}$, where $F_1, \ldots, F_{2n} \in {\rm{Pic}}(X)$;
\item[(b)] The underlying complex manifold $({\rm{U}}(E), J)$ admits a family of Hermitian metrics $\Omega(s), s \in (-\infty,\lambda)$, solving the geometric flow
\begin{equation}
\displaystyle \frac{\partial}{\partial s}\Omega(s)= -p(\Omega(s),t), \ \ \ \Omega(0) = \Omega_0,
\end{equation}
$\forall t \in \mathbb{R}$, such that $({\rm{U}}(E),\Omega_{0}) \to (X,\omega_{0})$ is a Hermitian submersion;
\item[(c)] For each $s \in (-\infty , \lambda)$, we have that $({\rm{U}}(E),\Omega(s) , J)$ is a balanced Hermitian manifold, i.e., $\delta \Omega(s) = 0$, $\forall s \in (-\infty , \lambda)$.
\end{enumerate}

\end{theorem}

\begin{proof}
Define $E = F_1 \oplus ... \oplus F_{2n}$, such that $F_i \in {\rm{Pic}}^{0}_{\omega_{0}}(X)$, $i = 1,\ldots,2n$, and fixed a Hermitian structure
\begin{equation*}
    h= h_1 \oplus...\oplus h_{2n}
\end{equation*}
 on $E$, where $h_i$ is a Hermitian structure on $F_i$, for $1 \leq i \leq 2n$. From this data, consider the principal $T^{2n}$-bundle ${\rm{U}}(E)$.\\

In the above setting, we can find a connection $\Theta \in \Omega^1 ({\rm{U}}(E),Lie(T^{2n}))$ 
\begin{equation*}
     \Theta= 
    \begin{bmatrix}
    \sqrt{-1} \Theta_1 &\hdots & 0\\
    \vdots & \ddots &\vdots\\
    0 &\hdots &\sqrt{-1}\Theta_{2n}
    \end{bmatrix}
\end{equation*}
satsifying the following properties (\cite{correa2023t}, Section 5.1)
\begin{enumerate}
    \item [(1)] $d\Theta_i = \pi^* (\psi_i )$ such that $\frac{\psi_i}{2\pi} \in c_1 ( F_i )$,
    \item[(2)] $\psi_j$ are $G$-invariant $(1,1)$-forms,
    \item[(3)] $\Lambda_{\omega_{0}} (\psi_j ) = 0$, $\forall j = 1,...,2n$.
\end{enumerate} 
As before, we can equip ${\rm{U}}(E)$ with the complex structure $J$ and Hermitian metric $\Omega$ defined as follows:
\begin{enumerate}
    \item [a)] $J|_{\ker(\Theta)}$ is given by the lift of the complex structure of $X$,
    \item[b)] $J\Theta_{2j-1} = -\Theta_{2j}$ and $J\Theta_{2j} = \Theta_{2j-1}$, $j=1,...,n$,
    \item[c)] $\Omega_0  = \pi^* ( \omega_{0} ) + \frac{1}{2}{\rm{tr}}(\Theta \wedge J\Theta)$. 
\end{enumerate}
From above, considering that K\"{a}hler-Ricci flow $\omega(s) =  (\lambda-s )p_{\omega_0}$ on $X$, starting at $\omega_{0} = \lambda p_{\omega_{0}}$, we obtain a one-parameter family of Hermitian metrics $\Omega_{s}$ on the complex manifold $({\rm{U}}(E),J)$, such that 
\begin{equation*}
   \Omega_s =\pi^* \big ( (\lambda-s )p_{\omega_0} \big ) +\frac{1}{2} {\rm{tr}}(\Theta \wedge J\Theta ).
\end{equation*}
From above, it follows that 
\begin{equation*}
  \frac{\partial}{\partial s} \Omega_{s} = \frac{\partial}{\partial s} \Big( \pi^* (\lambda-s )p_{\omega_0} +\frac{1}{2} {\rm{tr}}(\Theta \wedge J\Theta ) \Big) = -\pi^* p_{\omega_0 }.
\end{equation*}
Observing that $\rm{Pic}_{\omega_0 }^{0}(X ) = \rm{Pic}_{c\omega_0 }^{0}(X )$ and
 \begin{equation}
\Lambda_{c \omega_{0}}(-) = \frac{1}{c}\Lambda_{\omega_{0}}(-),
 \end{equation}
for every $c \in \mathbb{R}^+$, since $\Lambda_{\omega_0} (\psi_i )=0$, $\forall i = 1,\ldots,2n$, Equation \ref{tricciform} implies that
\begin{equation}
\label{ricci11type}
    p(\Omega_s ,t) = \pi^* p_{\omega_s}=\pi^* p_{\omega_0},  
\end{equation}
i.e., $\frac{\partial}{\partial s}\Omega_s =-p(\Omega_s ,t)$. Hence, we obtain items $(a)$ and $(b)$.\\

In order to prove item $(c)$, we notice that 
\begin{equation*}
 \delta\Omega_s = (\lambda-s)\pi^* (\delta \omega_0 ) =0,
\end{equation*}
see Equation \ref{balanced}. Therefore, $\Omega_s$ is balanced for every $s \in (-\infty , \lambda)$, which concludes the proof.
\end{proof}
Applying the result of the last theorem to $t=-1$ and using that Ricci-forms are all in this case of $(1,1)$-type (Equation \ref{ricci11type}), we obtain the following result:

\begin{Coro}
In the setting of Theorem \ref{TeoremaBa}, the family of Hermitian metrics $\Omega(s), s \in (-\infty,\lambda)$, on $({\rm{U}}(E), J)$, defines a solution to the pluriclosed flow
\begin{equation}
\displaystyle \frac{\partial }{\partial s} \Omega(s)= -\rho_{B}^{1,1}(\Omega(s)), \ \  \Omega(0)=\Omega_0,
\end{equation}
such that $({\rm{U}}(E),\Omega(s), J)$ is a balanced Hermitian manifold for all $s \in (-\infty , \lambda)$.
\end{Coro}

Now we wish to understand what happens to the flows given in Theorem \ref{Teorema2} and Theorem \ref{TeoremaBa} as $s$ approaches to $\lambda$, in the Gromov-Hausdorff sense. For more details on the Gromov-Hausdorff distance, see \cite[Chapter 7]{brin2001course}.\\

In order to prove our next result, let us introduce some terminology. Let $\pi \colon P \to (M,g_{M})$ be a principal $G$-bundle, such that $(M,g_{M})$ is a compact and connected Riemannian manifold and $G$ is a compact and connected Lie group. Fixed a principal connection $\Theta \in \Omega^1 (P, \mathfrak{g})$ and fixed a bi-invariant metric $\langle -,- \rangle$ on $G$, for every continuous function $f \colon [0,T) \to \mathbb{R}$, such that $T > 0$ and $f > 0$, we define the following Riemannian metrics on $P$
\begin{equation}
\label{GHmetricdef}
g_{t} := f(t)\pi^{\ast}g_{M} + \langle \Theta,\Theta \rangle. 
\end{equation}
Suppose now that ${\rm{Hol}}_{u}(\Theta) \subset G$ is a closed Lie subgroup. Given $u \in P$, we consider the Riemannian manifold $( G/{\rm{Hol}}_{u}(\Theta),g_{T})$, such that $g_{T}$ is the $G$-invariant Riemannian metric induced by $\langle -, - \rangle$ which makes the canonical projection
\begin{center}
$\pi \colon G \to G/{\rm{Hol}}_{u}(\Theta)$ 
\end{center}
a Riemannian submersion. In the above setting, we have the following result.

\begin{theorem}[\cite{correa2024bundle}]
\label{GHtheorem}
If $\lim_{t \rightarrow T}f(t) = 0$, then
\begin{equation}
\lim_{t \to T} d_{GH}\big ((P,d_{g_{t}}),(G/{\rm{Hol}}_{u}(\Theta),d_{g_{T}})\big)  = 0,
\end{equation}
where  $d_{GH}$ denotes the Gromov-Hausdorff distance.
\end{theorem}

Considering the $T^{2n}$-bundle ${\rm{U}}(E) \to X$ constructed in Theorem \ref{Teorema2} and Theorem \ref{TeoremaBa}, we notice that the family of Hermitian metrics 
\begin{equation}
     \Omega_s = (\lambda-s ) \pi^* p_{\omega_0} +\frac{1}{2} {\rm{tr}}(\Theta \wedge J\Theta ),
\end{equation}
such that $s \in (-\infty , \lambda)$, is a particular example of the Riemannian metric given in Equation \ref{GHmetricdef}. If we suppose that the holonomy group ${\rm{Hol}}_{u} (\Theta) \subset T^{2n}$ of the principal connection $\Theta$ is closed, since a closed and connected subgroup of a torus is also a torus, we have that ${\rm{Hol}}_{u} (\Theta) = T^{l}$, for some $0 \leq l \leq 2n$. Therefore, from Theorem \ref{GHtheorem}, we obtain the following result:

\begin{coro}\label{CoroF}
If the holonomy group ${\rm{Hol}}_{u} (\Theta)$ is closed in $T^{2n}$, then
\begin{equation*}
    \lim_{s \rightarrow \lambda} d_{GH} \big(({\rm{U}}(E) , d_{\Omega_s}), (T^{2n-l}, d_g )\big) =0,
\end{equation*}
where $g$ be the normal metric on $T^{2n-l}$ and $l = \dim{\rm{Hol}}_{u} (\Theta)$.
\end{coro}

As we can see, the result above provides a description for the limiting behavior of the geometric flows of Hermitian metrics obtained in in Theorem \ref{Teorema2} and Theorem \ref{TeoremaBa}.

\section{Examples and final comments}

In what follows, we present some concrete examples which illustrate our main results. Throughout this section we consider the notations and conventions of Example \ref{exampleprojtangent}. We start by presenting examples of Chern-Ricci flow which do not preserve the $t$-Gauduchon Ricci-flat condition for $t < 1$.

\begin{example}[Chern-Ricci flow] Consider $X = {\mathbb{P}}(T_{{\mathbb{P}^{2}}}) = {\rm{SL}}_{3}(\mathbb{C})/B$. In this case, we have the Fano index ${\rm{I}}({\mathbb{P}}(T_{{\mathbb{P}^{2}}})) = 2$ and 
\begin{equation}
\mathcal{O}_{{\mathbb{P}}(T_{{\mathbb{P}^{2}}})}(1) = -\frac{1}{2}K_{{\mathbb{P}}(T_{{\mathbb{P}^{2}}})} = \mathscr{O}_{\alpha_{1}}(1) \otimes \mathscr{O}_{\alpha_{2}}(1).
\end{equation}
Considering the unique ${\rm{SU}}(3)$-invariant K\"{a}hler metric $\vartheta_{0} \in c_{1}({\mathbb{P}}(T_{{\mathbb{P}^{2}}}))$, such that $\vartheta_{0} = 2({\bf{\Omega}}_{\alpha_{1}} + {\bf{\Omega}}_{\alpha_{2}})$, let
\begin{equation}
{\rm{Pic}}^{0}_{\vartheta_{0}}({\mathbb{P}}(T_{{\mathbb{P}^{2}}})) = \big \langle \mathscr{O}_{\alpha_{1}}(-1)\otimes \mathscr{O}_{\alpha_{2}}(1) \big \rangle.
\end{equation}
For every $\ell \in \mathbb{Z}^{\times}$, we denote 
\begin{equation}
F_{\ell}:=  \mathscr{O}_{\alpha_{1}}(-\ell)\otimes \mathscr{O}_{\alpha_{2}}(\ell). 
\end{equation}
From above, for every $k,\ell \in \mathbb{Z}^{\times}$, we set
\begin{equation}
E_{k,\ell}:= \mathcal{O}_{{\mathbb{P}}(T_{{\mathbb{P}^{2}}})}(k) \oplus F_{\ell}. 
\end{equation}
Observing that $p_{\vartheta_{0}} = \vartheta_{0}$, for every $\lambda > 0$, we have a K\"{a}hler-Einstein metric $\omega_{0} = \lambda\vartheta_{0}$ on ${\mathbb{P}}(T_{{\mathbb{P}^{2}}})$. From these data, we have a principal $T^{2}$-bundle
\begin{equation}
T^{2} \hookrightarrow {\rm{U}}(E_{k,\ell}) \to {\mathbb{P}}(T_{{\mathbb{P}^{2}}}),
\end{equation}
and we can take a principal connection 
\begin{equation}
\label{pconnectiobundle}
\Theta = \begin{pmatrix} \sqrt{-1}\Theta_{1} & 0 \\
0 & \sqrt{-1}\Theta_{2}\end{pmatrix} \in \Omega({\rm{U}}(E_{k,\ell}),{\rm{Lie}}(T^{2})),
\end{equation}
such that 
\begin{equation}
d\Theta_{1} = \pi^{\ast}\big (k ({\bf{\Omega}}_{\alpha_{1}} + {\bf{\Omega}}_{\alpha_{2}})\big) \ \ \text{and} \ \ d\Theta_{2} = \pi^{\ast}\big (\ell ({\bf{\Omega}}_{\alpha_{2}} - {\bf{\Omega}}_{\alpha_{1}})\big).
\end{equation}
Now we can equip ${\rm{U}}(E_{k,\ell})$ with a 1-parameter family of Hermitian structure $(\Omega(s),J)$, such that 
\begin{equation}
\Omega(s) = (\lambda - s) \pi^{\ast}(\vartheta_{0}) + \Theta_{1} \wedge \Theta_{2},
\end{equation}
such that $s \in (-\infty,\lambda)$. From Theorem \ref{TeoremaA}, it follows that $\Omega(s)$ is a solution for the Chern-Ricci flow
\begin{equation}
\displaystyle \frac{\partial}{\partial s}\Omega(s)= -p(\Omega(s),1), \ \ \ \Omega(0) = \lambda \pi^{\ast}(\vartheta_{0}) + \Theta_{1} \wedge \Theta_{2}.
\end{equation}
Moreover, following Theorem \ref{TeoremaA}, we conclude that $\Omega(s)$ is $t$-Gauduchon Ricci-flat if, and only if,
\begin{equation*}
  t =1 - \frac{2(\lambda-s) {\rm{I}}({\mathbb{P}}(T_{{\mathbb{P}^{2}}}))^2}{k^2 \dim_{\mathbb{C}}({\mathbb{P}}(T_{{\mathbb{P}^{2}}}) )} = 1 - \frac{8}{3} \frac{(\lambda - s)}{k^{2}}.
\end{equation*}
In particular, we have the following cases
\begin{enumerate}
\item[(i)]$t=-1$: $\Omega(s)$ is Strominger-Bismut Ricci-flat if, and only if, $s = \lambda  - \frac{3}{4} k^{2}$;
\item[(ii)] $t=0$: $\Omega(s)$ is Lichnerowicz Ricci-flat if, and only if, $s = \lambda - \frac{3}{8}k^{2}$.
\end{enumerate}
It is worth pointing out that, for a fixed $k \in \mathbb{Z}^{\times}$, one can choose the Einstein constant $\lambda$ in a suitable way to produce examples of Chern-Ricci flow with initial condition satisfying the $t$-Gauduchon Ricci-flat equation.
\end{example}

Now we present an example which illustrates the results of Theorem \ref{TeoremaB} and Corollary \ref{Balanced}.

\begin{example}[Pluriclosed flow]
 Considering $X = {\mathbb{P}}(T_{{\mathbb{P}^{2}}}) = {\rm{SL}}_{3}(\mathbb{C})/B$ as before, for $a,b \in \mathbb{Z}^{\times}$, we take
 \begin{equation}
F_{a} = \mathscr{O}_{\alpha_{1}}(-a)\otimes \mathscr{O}_{\alpha_{2}}(a) \ \ \ \text{and} \ \ \ F_{b} = \mathscr{O}_{\alpha_{1}}(-b)\otimes \mathscr{O}_{\alpha_{2}}(b).
 \end{equation}
From above, we construct the rank  $2$ holomorphic vector bundle
\begin{equation}
E_{a,b} = F_{a} \oplus F_{b} \to {\mathbb{P}}(T_{{\mathbb{P}^{2}}}).
\end{equation}
Similarly to the previous example, we consider the principal $T^{2}$-bundle
\begin{equation}
T^{2} \hookrightarrow {\rm{U}}(E_{a,b}) \to {\mathbb{P}}(T_{{\mathbb{P}^{2}}}),
\end{equation}
equipped with a principal connection 
\begin{equation}
\Theta = \begin{pmatrix} \sqrt{-1}\Theta_{1} & 0 \\
0 & \sqrt{-1}\Theta_{2}\end{pmatrix} \in \Omega({\rm{U}}(E_{a,b}),{\rm{Lie}}(T^{2})),
\end{equation}
satisfying
\begin{equation}
d\Theta_{1} = \pi^{\ast}\big (a ({\bf{\Omega}}_{\alpha_{2}} - {\bf{\Omega}}_{\alpha_{1}})\big) \ \ \text{and} \ \ d\Theta_{2} = \pi^{\ast}\big (b ({\bf{\Omega}}_{\alpha_{2}} - {\bf{\Omega}}_{\alpha_{1}})\big).
\end{equation}
From Theorem \ref{TeoremaB}, it follows that ${\rm{U}}(E_{a,b})$ admits a $1$-parameter family of Hermitian structures $(\Omega(s),J)$, such that \begin{equation}
\Omega(s) = (\lambda - s) \pi^{\ast}(\vartheta_{0}) + \Theta_{1} \wedge \Theta_{2},
\end{equation}
with $s \in (-\infty,\lambda)$, is a solution for the geometric flow
\begin{equation}
\displaystyle \frac{\partial}{\partial s}\Omega(s)= -p(\Omega(s),t), \ \ \ \ \Omega(0) = \Omega_0,
\end{equation}
for all $t \in \mathbb{R}$. Moreover, we have in this case that
\begin{equation}
\delta \Omega(s) = 0, \ \ \forall s \in (-\infty,\lambda),
\end{equation}
i.e., $({\rm{U}}(E_{a,b}),\Omega(s),J)$ is a balanced Hermitian manifold, $\forall s \in (-\infty,\lambda)$.\\

In the above setting, one can also verify that 
\begin{equation}
p(\Omega(s),t) = p_{1}(\Omega(s),t),  \ \ \ \ \forall t \in \mathbb{R},
\end{equation}
see Eq. (\ref{(1,1)partGauduchon})). Thus, considering $\rho_{B}^{1,1}(\Omega(s)) = p_{1}(\Omega(s),-1)$, we obtain a solution to the pluriclosed flow 
\begin{equation}
\displaystyle \frac{\partial }{\partial s} \Omega(s)= -\rho_{B}^{1,1}(\Omega(s)), \ \ \ \ \Omega(0)=\Omega_0,
\end{equation}
such that $\delta \Omega(s) = 0$ for every $s \in (-\infty,\lambda)$.\\

We notice that, since ${\rm{U}}(E_{a,b})$ can not carry a K\"{a}hler metric for purely topological reasons, see for instance \cite{hofer1993remarks} and \cite{poddar2018group}, the Hermitian metric $\Omega(s)$ described above cannot be pluriclosed $\forall s \in (-\infty,\lambda)$ (e.g. \cite{popovici2013aeppli}). To the best of the author's knowledge, it is the first example in the literature of balanced non-pluriclosed solution to the pluriclosed flow.
\end{example}

Our last example illustrates the result of Corollary \ref{GHHermitian} by means of the pluriclosed flow presented in the above example.

\begin{example}[Gromov-Hausdorff convergence] Consider the pluriclosed flow $({\rm{U}}(E_{a,b}),\Omega(s))$ as in the previous example. Given $u \in {\rm{U}}(E_{a,b})$, let
\begin{equation}
P_{u}(\Theta) = \Big \{ x \in {\rm{U}}(E_{a,b}) \ \Big | \ x \sim u  \Big \}, 
\end{equation}
be the holonomy bundle of $\Theta$ (as in Eq. (\ref{pconnectiobundle})) at $u \in {\rm{U}}(E_{a,b})$, i.e., the set of points which can be joined to $u$ by a $\Theta$-horizontal path. By definition, the Lie algebra $\mathfrak{hol}_{u}(\Theta)$ of the holonomy group ${\rm{Hol}}_{u}(\Theta)$ is given by
\begin{equation}
\mathfrak{hol}_{u}(\Theta) = {\rm{Span}}_{\mathbb{R}}\big \{  d\Theta_{x}\big (V,W\big ) \ \  | \ \ x \in P_{u}(\Theta), \ \ V,W \in \ker(\Theta)_{x}\big\}.
\end{equation}
Now we notice that 
\begin{equation}
d\Theta =  \begin{pmatrix} \sqrt{-1}d\Theta_{1} & 0 \\
0 & \sqrt{-1}d\Theta_{2}\end{pmatrix} = \begin{pmatrix} \sqrt{-1}a\pi^{\ast}(\psi) & 0 \\
0 & \sqrt{-1}b\pi^{\ast}(\psi)\end{pmatrix},
\end{equation}
such that 
\begin{equation}
\psi = {\bf{\Omega}}_{\alpha_{2}} - {\bf{\Omega}}_{\alpha_{1}} \in \Omega^{1,1}({\mathbb{P}}(T_{{\mathbb{P}^{2}}})).
\end{equation}
Hence, we obtain the following description for the holonomy Lie algebra 
\begin{equation}
\mathfrak{hol}_{u}(\Theta) = {\rm{Span}}_{\mathbb{R}}\Bigg \{ \begin{pmatrix} \sqrt{1}r & 0 \\
0 & \sqrt{-1}r\frac{b}{a}\end{pmatrix} \ \Bigg | \ r \in \mathbb{R}\Bigg\}.
\end{equation}
From above, we can describe the holonomy group of $\Theta$ at $u \in {\rm{U}}(E_{a,b})$ as follows
\begin{equation}
{\rm{Hol}}_{u}(\Theta) = \Bigg \{ \begin{pmatrix} {\rm{e}}^{\sqrt{1}r} & 0 \\
0 & {\rm{e}}^{\sqrt{-1}r\frac{b}{a}}\end{pmatrix} \ \Bigg | \ r \in \mathbb{R}\Bigg\} \subset T^{2}.
\end{equation}
Observing that $a,b \in \mathbb{Z}^{\times}$ and that ${\mathbb{P}}(T_{{\mathbb{P}^{2}}})$ is simply connected, it follows that ${\rm{Hol}}_{u}(\Theta) \subset T^{2}$ is a closed connected Lie subgroup. Thus, from Corollary \ref{CoroF}, it follows that 
\begin{equation*}
    \lim_{s \rightarrow \lambda} d_{GH} \big(({\rm{U}}(E_{a,b}) , d_{\Omega_s}), (T^{2}/{\rm{Hol}}_{u}(\Theta), d_g )\big) =0,
\end{equation*}
that is, the pluriclosed flow $({\rm{U}}(E_{a,b}),\Omega(s))$ converges to $(T^{2}/{\rm{Hol}}_{u}(\Theta)) = (S^{1},g)$, in the Gromov-Hausdorff sense, as $s \to \lambda$.
\end{example}


\bibliographystyle{alpha}
\bibliography{ref.bib}

\end{document}